\documentclass[12pt]{amsart}

\usepackage{amsmath,amstext,amssymb,amsopn,amsthm}
\usepackage{url,verbatim}
\usepackage{mathtools}
\usepackage{enumerate}

\usepackage{color,graphicx}

\usepackage[margin=30mm]{geometry}
\usepackage{eucal,mathrsfs,dsfont}

\usepackage{soul}

\newtheorem{theorem}{Theorem}[section]
\newtheorem{corollary}[theorem]{Corollary}
\newtheorem{lemma}[theorem]{Lemma}
\newtheorem{proposition}[theorem]{Proposition}

\theoremstyle{definition}

\newtheorem{definition}[theorem]{Definition}
\newtheorem{problem}[theorem]{Problem}

\theoremstyle{remark}
\newtheorem{remark}[theorem]{Remark}

\numberwithin{equation}{section}

\newcommand{\eps}{\varepsilon}

\newcommand{\calL}{\mathcal{L}}

\newcommand{\calA}{\mathcal{A}}
\newcommand{\calD}{\mathcal{D}}

\newcommand{\calT}{\mathcal{T}}
\newcommand{\calS}{\mathcal{S}}
\newcommand{\calW}{\mathcal{W}}

\newcommand{\calC}{\mathcal{C}}
\newcommand{\calB}{\mathcal{B}}
\newcommand{\calV}{\mathcal{V}}
\newcommand{\calH}{\mathcal{H}}

\newcommand{\calM}{\mathcal{M}}

\newcommand{\R}{\mathds{R}}
\newcommand{\C}{\mathds{C}}

\newcommand{\Z}{{\mathbb Z}}

\newcommand{\prt}{\partial}

\newcommand{\ol}{\overline}
\newcommand{\wh}{\widehat}
\newcommand{\wt}{\widetilde}
\newcommand{\om}{\ominus}
\newcommand{\bv}{\mathbf{v}}
\newcommand{\bw}{\mathbf{w}}
\newcommand{\bs}{\mathbf{s}}

\DeclareMathOperator{\Leb}{Leb}

\DeclareMathOperator{\mtwo}{{\bf \Lambda}}

\def\bv{{\bf v}}

\def\ball{\mathbf{B}}

\title[Invisibility via reflecting coating]{Invisibility via reflecting coating}
\author{Krzysztof Burdzy and Tadeusz Kulczycki}

\address{KB: Department of Mathematics, Box 354350, University of Washington, Seattle, WA 98195}
\email{burdzy@math.washington.edu}
\thanks{K.~Burdzy was supported in part by NSF Grant DMS-0906743 and
by grant N N201 397137, MNiSW, Poland. }
\address{TK: Institute of Mathematics and Computer Science, Wroc{\l}aw University of Technology, Wybrzeze Wyspianskiego 27, 50-370 Wroc{\l}aw, Poland}
\email{Tadeusz.Kulczycki@pwr.wroc.pl}

\begin{document}

\begin{abstract}
We construct a subset $A$ of the unit disc with the following properties. (i) The set $A$ is the finite union of disjoint  line segments. (ii) The shadow of $A$ is arbitrarily close to the shadow of the unit disc in ``most'' directions. (iii) If the line segments are considered to be mirrors reflecting light according to the classical law of specular reflection then most light rays hitting the set emerge on the other side of the disc moving along a parallel line and  shifted by an arbitrarily small amount.

We also construct a set which reflects almost all light rays coming from one direction to another direction but 
its shadow is arbitrarily small in other directions, except for an arbitrarily small family of directions. 
\end{abstract}

\maketitle

\section{Introduction and the main result}\label{intro}

We will construct two pre-fractal sets with special reflective and projection properties. First, we will construct a subset $A$ of the unit disc with the following properties. (i) The set $A$ is the finite union of disjoint   line segments. (ii) The shadow of $A$ is arbitrarily close to the shadow of the unit disc in most directions (that is, the complement of the set of such directions has arbitrarily small measure). (iii) If the line segments are considered to be mirrors reflecting light according to the classical law of specular reflection then most light rays hitting the set emerge on the other side of the disc traveling along a parallel line  and shifted by an arbitrarily small amount.

Next, we will construct a set which reflects almost all light rays coming from one direction to another direction but 
its shadow is arbitrarily small in other directions, except for an arbitrarily small family of directions. 

The article has multiple sources of inspiration.
One of them is recent progress on invisibility (see \cite{uhl1,uhl2}) although there is no direct relationship between our article and those papers at the technical level.
Another source of inspiration is the theory of radiative transfer (see \cite{peraiah}) which is used by astrophysicists to study the scattering of light.
The relationship between reflections and ``visibility'' is implicit in Problem 2.6 of \cite{oakbss}.
And last but not least, Falconer's ``digital sundial'' theorem (\cite[Thm. 6.9]{falcbook}; see also Proposition \ref{m15.1} below) provided not only inspiration but also a significant mathematical step in our argument.

The set constructed in Theorem \ref{a3.1} below might be a building block of a reflective ``surface'' 
that is the subject of Problem 2.6 of \cite{oakbss}. At this point, we do not know how to turn this observation into a rigorous solution of that problem. 
The set in Theorem \ref{a3.1} resembles the set in the Besicovitch theorem 
(see \cite[Thm.~6.15, p.~90]{falcgeo}) and the ``four-corner Cantor set'' 
(see \cite{PSS,BV,NPV}) but we did not find a way to turn these
classical constructions into a direct proof of our result.

We now introduce notation and definitions needed to state our
results in a rigorous way. It will be convenient to identify
$\R^2$ and $\C$ and switch between the vector and complex
analytic notation. Let $\Pi_\theta A$ denote the orthogonal
projection of a set $A\subset \R^2 $ on the line $K_\theta :=
\{z=re^{i\theta}: r\in \R\}$. The 1-dimensional Lebesgue
measure will be denoted $\Leb$. By abuse of notation, we will
use the same symbol $\Leb$ to denote 1-dimensional measures on
lines in $\R^2$ and restriction of the Lebesgue measure to
$[0,2\pi)$. Let $D = \{x\in \R^2: |x|\leq 1\}$ denote the unit
disc.

We will be concerned with light rays moving at a constant speed and reflecting from ``mirrors,'' that is, line segments. We will describe trajectories of light rays using ``directed lines''. A \emph{directed line} $L_{\bv,\bw}$ is an affine map $\Gamma_{\bv,\bw}: \R \to \R^2$,
$\Gamma_{\bv,\bw}(t) = t\bv + \bw$, where $\bv,\bw\in \R^2$, $\bv \perp \bw$ and $|\bv| =1$.
We will write $L_{\bv,\bw}=\Gamma_{\bv,\bw}(\R)$.
Let $\calW$ be the collection of all pairs $(\bv,\bw)$ such that
$\bv,\bw\in \R^2$, $\bv \perp \bw$ and $|\bv| =1$. The set $\calW$ may be considered to be a subset of $\R^4$. As such, it inherits the usual metric and topology from $\R^4$.
The following formula uniquely defines a measure $\mtwo$ on $\calW$,
\begin{align}\label{m18.5}
&\mtwo(\{(\bv,\bw): \bw = r e^{i\theta}, \bv = e^{i(\theta+\pi/2)} , \theta\in [\theta_1,\theta_2], r\in[r_1,r_2]\}) \\
&=\mtwo(\{(\bv,\bw): \bw = r e^{i\theta}, \bv = e^{i(\theta-\pi/2)} , \theta\in [\theta_1,\theta_2], r\in[r_1,r_2]\}) \nonumber \\
&= (\theta_2-\theta_1) (r_2-r_1).\nonumber
\end{align}
Let $\calV$ be the set of all pairs $(\bv,\bw)\in\calW$ such
that $\bw \in D$. Then $\mtwo(\calV) = 4\pi$. Heuristically
speaking, according to the probability measure
$\mtwo(\,\cdot\,)/(4\pi)$, the direction of $\Gamma_{\bv,\bw}$ is
chosen uniformly, and so is the distance from $0$ (within $[0,1]$).

A \emph{light ray} is a continuous and piecewise affine map $R: \R \to \R^2$. We require that for each light ray, $\R$ can be partitioned into a finite number of intervals
and on each interval, $R(t) = (t-t_1) \bv + \bw$, for some $(\bv, \bw)$ and $t_1$ depending on the interval. Either one or two of the intervals have infinite length. We will say that $R$ comes from direction
$\Gamma_{\bv_1, \bw_1}$ if $R(t) = (t-t_1) \bv_1 + \bw_1$ on the interval that extends to $-\infty$. We will say that $R$ escapes in direction
$\Gamma_{\bv_2, \bw_2}$ if $R(t) = (t-t_2) \bv_2 + \bw_2$ on the interval that extends to $+\infty$.
We will assume that light rays reflect from mirrors (line segments) according to the classical law of specular reflection in which the angle of reflection is equal to the angle of incidence.
Suppose that a set $F\subset \R^2$ consists of a finite number of line segments. Then a light ray which arrives ``from infinity'' and hits $F$ will not be trapped (this can be proved as Proposition 2.1 in \cite{oakbss}) so it has a well defined direction from which it arrives and a direction
of escape. 
If a light ray which comes from direction
$\Gamma_{\bv_1, \bw_1}$ escapes in the direction $\Gamma_{\bv_2, \bw_2}$
after reflecting in mirrors comprising $F$ then we 
will write $\calT_F(\bv_1, \bw_1) = (\bv_2, \bw_2)$.

The following is the first of our two main results.

\begin{theorem}\label{m16.10}(Invisibility via reflecting coating)
For every $\eps>0$ there exists a set $F\subset D$ which
consists of a finite number of line segments and such that
there exists a set $\calA \subset \calV$ with $\mtwo(\calV
\setminus \calA) \leq \eps$, satisfying the following
conditions.
\begin{enumerate}[(i)]
\item For all $(\bv,\bw) \in \calA$, the line
    $L_{\bv,\bw}(\R)$ intersects $F$.
\item For all $(\bv,\bw) \in \calA$, we have
    $\calT_F(\bv,\bw) = (\bv, \bw_1)$ for some $\bw_1$
    satisfying
\begin{align}\label{a4.1}
|\bw - \bw_1| \leq\eps.
\end{align}
\end{enumerate}
\end{theorem}

Note that it is the same direction $\bv$ on both sides of the formula $\calT_F(\bv,\bw) = (\bv, \bw_1)$ in Theorem \ref{m16.10} (ii).

\begin{problem}
Does Theorem \ref{m16.10} remain true if \eqref{a4.1} is replaced with $|\bw - \bw_1| =0$? 
\end{problem}

Our second result is the following.

\begin{theorem}\label{a3.1} (Invisible mirror)
For every $\theta \in (0,\pi)$ and $\eps >0$
there exists a set $G \subset \R^2$ which
consists of a finite number of line segments and has the following properties. Let $\calB =
\{(\bv,\bw): \bw = r e^{i(\theta+\pi/2)}, \bv = e^{i\theta} , r\in[0,1]\}$.
\begin{enumerate}[(i)]
\item For 
$\Leb$-almost all $r\in[0,1]$, the light ray arriving from the direction $\Gamma_{(1,0),(0,r)}$ reflects only once in a mirror in $G$ and $\calT_G((1,0), (0,r)) \in \calB$. 
\item For all $\alpha$ satisfying $\alpha -\pi/2 \in [\eps, \theta-\eps] \cup [\theta+\eps, \pi -\eps]$ we have
$\Leb( \Pi_\alpha G) < \eps$.
\end{enumerate}
\end{theorem}

The rest of the paper is organized as follows. We will construct the set $F$ described in Theorem \ref{m16.10} as the union of a large number of very small fractal-like sets. We will call these sets blocks. The construction of an individual block will be based on a Falkoner's argument and it will be presented in Section \ref{blocks}. We will assemble blocks into a star-shaped structure (set $F$) in Section \ref{hedge}. We will analyze the shadow of $F$ in the same section. Section \ref{reflect} will be devoted to the analysis of light reflections in the set $F$.
Section \ref{invmir} will be devoted to the proof of Theorem \ref{a3.1}.

\bigskip

We are grateful to Donald Marshall and Boris Solomyak for very helpful advice.

\section{Building blocks}\label{blocks}

Let $I_\theta = \{z=re^{i\theta}: r\in [0,1]\}$.
The symmetric difference of sets $A$ and $B$ will be denoted $A \om B$.

The following result is a special case of Falconer's ``digital sundial'' theorem. See \cite[Thm. 6.9]{falcbook}
for an accessible statement of the result and a sketch of the proof. A rigorous proof was given in \cite{falcart} in the multidimensional setting.

\begin{proposition}\label{m15.1}
For every $\theta_1 \in (0,\pi)$, there exists a measurable set $A\subset \R^2$ with the following properties.
\begin{enumerate}[(i)]
\item For $\Leb$-almost all $\theta \in [0, \theta_1]$,
\begin{align*}
\Leb (\Pi_\theta A \om I_\theta ) = 0.
\end{align*}
\item For $\Leb$-almost all $\theta \in ( \theta_1, \pi)$,
\begin{align*}
\Leb (\Pi_\theta A  ) = 0.
\end{align*}
\end{enumerate}
\end{proposition}
In words, the projection of $A$ in almost every direction
$\theta \in [0, \theta_1]$ is ``indistinguishable'' from a line segment of unit length, while the projection in
almost every direction $\theta \in ( \theta_1, \pi)$ is ``almost'' invisible (has measure zero).

The next proposition is the main step in the construction of our basic building block.
Although this result is inspired by and very close to, heuristically speaking, Proposition \ref{m15.1}, it does not follow directly from Proposition \ref{m15.1}. Our proof will be based on a claim embedded in a proof of a theorem in \cite{falcart}.

We will call a set a \emph{diamond} if it is a closed square with sides inclined at the angle $\pi/4$ to the axes. We will say that $A\subset \R^2$ is a \emph{diamond set} if it is a finite union of diamonds.
Let $Q_\eps$ denote the square $(-\eps, 1+\eps)^2$.

\begin{proposition}\label{m15.2}
For every $\eps>0$ and $\theta_1\in(0,\pi/2)$
there exists a diamond set $A\subset Q_\eps$ which satisfies the following.
\begin{enumerate}[(i)]
\item There exists $\Theta_1 \subset [0, \theta_1]$ such
    that $\Leb([0,\theta_1] \setminus \Theta_1) \leq \eps$
    and for all $\theta \in \Theta_1$,
\begin{align*}
\Leb (\Pi_\theta A \om I_\theta ) \leq\eps.
\end{align*}
\item There exists $\Theta_2 \subset ( \theta_1,\pi)$ such
    that $\Leb((\theta_1,\pi) \setminus \Theta_2) \leq
    \eps$ and for all $\theta \in \Theta_2$,
\begin{align*}
\Leb (\Pi_\theta A  ) \leq\eps.
\end{align*}
\end{enumerate}
\end{proposition}

\begin{proof}
The following claim follows from the assertion made in the part of the proof of Theorem 5.1 in \cite{falcart} on page 58 and the first paragraph on page 59 (the assertion is proved later in the same proof). For any $\eps>0$ there exists a countable family of
open balls $\{\ball_k\}_{k\geq 1}$ such that $B := \bigcup _{k\geq 1} \ball_k  \subset Q_\eps$, and
\begin{align}
&I_\theta \subset \Pi_\theta B, \qquad \text{  for  }
 \theta \in [0, \theta_1], \label{m15.3} \\
& \int_0^{\theta_1} \Leb (\Pi_\theta B \setminus I_\theta )d\theta
+ \int_{\theta_1}^\pi \Leb (\Pi_\theta B )d\theta\leq \eps^2/8.\label{m15.4}
\end{align}
Let $B_n = \bigcup _{1\leq k\leq n} \ball_k$.
Then \eqref{m15.4} implies that
\begin{align*}
& \int_0^{\theta_1} \Leb (\Pi_\theta B_n \setminus I_\theta )d\theta
+ \int_{\theta_1}^\pi \Leb (\Pi_\theta B_n )d\theta \leq  \eps^2/8.
\end{align*}

Let $\Lambda_2^n = \{\theta \in (\theta_1, \pi): \Leb (\Pi_\theta
B_n) \leq \eps\}$. Then
\begin{align*}
\eps \Leb((\theta_1,\pi) \setminus \Lambda_2^n)
\leq \int_{\theta_1}^\pi \Leb (\Pi_\theta B_n )d\theta \leq  \eps^2/8.
\end{align*}
It follows that for any $n$,
\begin{align}\label{m15.7}
\Leb((\theta_1,\pi) \setminus \Lambda_2^n) \leq \eps/8.
\end{align}
Similarly, if $\wt\Lambda_1^n = \{\theta \in [0,\theta_1]: \Leb
(\Pi_\theta B_n \setminus I_\theta) \leq \eps/2\}$ then for any
$n$,
\begin{align*}
(\eps/2) \Leb([0,\theta_1] \setminus \wt\Lambda_1^n )
\leq \int_0^{\theta_1} \Leb (\Pi_\theta B_n\setminus I_\theta )d\theta \leq  \eps^2/8,
\end{align*}
and
\begin{align}\label{m15.5}
\Leb( [0,\theta_1] \setminus \wt\Lambda_1^n) \leq \eps/4.
\end{align}

The sequence of functions $f_n(\theta) : = \Leb
(I_\theta\setminus\Pi_\theta B_n  )$ is monotone and converges
to 0 in view of \eqref{m15.3}. Hence,
\begin{align*}
\lim_{n\to \infty} \int_0^{\theta_1} \Leb (I_\theta\setminus\Pi_\theta B_n )d\theta =0,
\end{align*}
and, therefore, there exists $m$ such that
\begin{align}\label{m15.8}
 \int_0^{\theta_1} \Leb (I_\theta\setminus\Pi_\theta B_m )d\theta \leq \eps^2/8.
\end{align}
Fix an $m$ satisfying this estimate for the rest of the proof.

Let $\calD_k$ be the collection of all diamonds with diameters
of length $2^{-k+1}$ and vertices in the lattice $2^{-k}\Z^2$.
Let $A_k$ be the union of all diamonds that belong to $\calD_k$
and are subsets of $B_m$. Let $\Theta_2^k = \{\theta \in
(\theta_1, \pi): \Leb (\Pi_\theta A_k) \leq \eps\}$. It follows
from \eqref{m15.7} that for any $k$,
\begin{align}\label{m15.9}
\Leb((\theta_1,\pi) \setminus \Theta_2^k) \leq \eps/8.
\end{align}

Let $\wt\Theta_1^k = \{\theta \in [0,\theta_1]: \Leb
(\Pi_\theta A_k \setminus I_\theta) \leq \eps/2\}$. Then for
any $k$, by \eqref{m15.5},
\begin{align}\label{m15.10}
\Leb( [0,\theta_1] \setminus \wt\Theta_1^k) \leq \eps/4.
\end{align}

The sequence of functions $g_k(\theta) : = \Leb
(I_\theta\setminus\Pi_\theta A_k  )$ is monotone and converges
to $\Leb (I_\theta\setminus\Pi_\theta B_m  )$. Hence,
\begin{align*}
\lim_{k\to \infty} \int_0^{\theta_1} \Leb (I_\theta\setminus\Pi_\theta A_k )d\theta =
\int_0^{\theta_1} \Leb (I_\theta\setminus\Pi_\theta B_m )d\theta.
\end{align*}
This and \eqref{m15.8} imply that there exists $k$ such that
\begin{align*}
 \int_0^{\theta_1} \Leb (I_\theta\setminus\Pi_\theta A_k )d\theta \leq \eps^2/4.
\end{align*}
We let $\wh \Theta_1^k = \{\theta \in  [0,\theta_1] : \Leb
(I_\theta\setminus\Pi_\theta A_k ) \leq \eps/2\}$. The we can
argue as above that
\begin{align}\label{m15.6}
\Leb([0,\theta_1] \setminus \wh \Theta_1^k) \leq \eps/2.
\end{align}
Fix a $k$ satisfying this estimate. Let $\Theta_1^k =  \{\theta
\in  [0,\theta_1]  : \Leb (\Pi_\theta A_k \om I_\theta) \leq
\eps\}$. We combine \eqref{m15.10} and \eqref{m15.6} to see
that
\begin{align}\label{m15.11}
\Leb([0,\theta_1] \setminus  \Theta_1^k) \leq \eps.
\end{align}
We now let $A = A_k$, $\Theta _1 = \Theta_1^k$ and $\Theta _2 = \Theta_2^k$. These sets satisfy the proposition in view of  \eqref{m15.9} and \eqref{m15.11}.
\end{proof}

Let $I_{\theta,\rho} = \{z=re^{i\theta}: r\in [0,\rho]\}$
and $Q_{\eps,\rho} = (-\rho\eps, \rho(1+\eps))^2$.
The following corollary follows from Proposition \ref{m15.2} by scaling.

\begin{corollary}\label{m16.1}
For all $\rho>0$, $\eps\in(0,1)$ and $\theta_1\in(0,\pi/2)$
there exists a diamond set $A\subset Q_{\eps,\rho}$ which has the following properties.
\begin{enumerate}[(i)]
\item There exists $\Theta_1 \subset [0, \theta_1]$ such
    that $\Leb([0,\theta_1] \setminus \Theta_1) \leq \eps$
    and for all $\theta \in \Theta_1$,
\begin{align*}
\Leb (\Pi_\theta A \om I_{\theta,\rho} ) \leq\eps\rho.
\end{align*}
\item There exists $\Theta_2 \subset ( \theta_1,\pi)$ such
    that $\Leb((\theta_1,\pi) \setminus \Theta_2) \leq
    \eps$ and for all $\theta \in \Theta_2$,
\begin{align*}
\Leb (\Pi_\theta A  ) \leq\eps\rho.
\end{align*}
\end{enumerate}
\end{corollary}

\bigskip

\begin{definition}\label{m16.3}
Consider a diamond set $A$ satisfying Corollary \ref{m16.1} for some $\rho>0$, $\eps\in(0,1)$ and $\theta_1\in(0,\pi/2)$. The set $A$ is a finite union of diamonds $\{G_k\}_{1\leq k \leq n}$. We can label the vertices $x^k_1, x^k_2, x^k_3, x^k_4$ of the diamond $G_k$ so that they satisfy $ x^k_2= x^k_1+ r_k (1,1), x^k_3 =x^k_1 + r_k (1,-1)$ and $x^k_4= x^k_1+r_k(2,0)$, for some $r_k>0$ and $x^k_1\in \R^2$.
Let $C_k$ be the union of three closed line segments
$\ol{ x^k_1, x^k_3}$, $\ol{ x^k_2, x^k_4}$ and  $\ol{ (x^k_1+ x^k_2)/2 , (x^k_3+x^k_4)/2}$
(see Fig.~\ref{fig1}).
\begin{figure}
\centering
 \includegraphics[width=5cm]{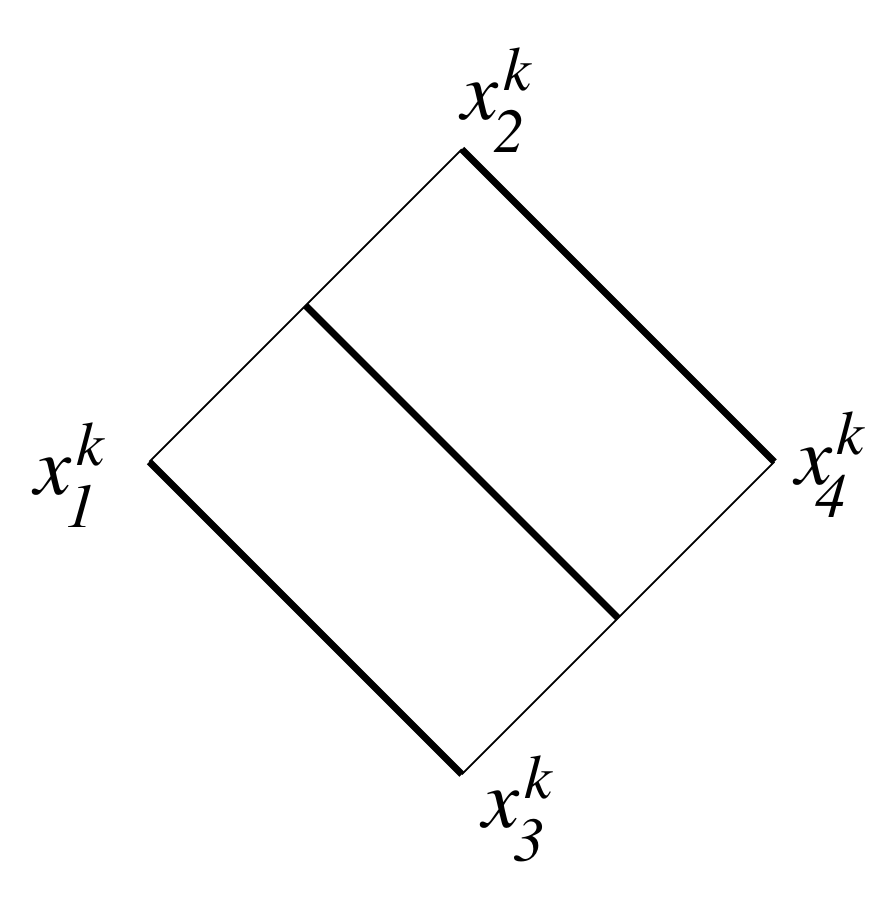}
\caption{A diamond and the three associated line segments.}\label{fig1}
\end{figure}
We let $S =S_{\rho,\theta_1, \eps}= \bigcup_{1\leq k \leq n} C_k$. We will call the set $S$ a \emph{block}.
\end{definition}

\begin{remark}\label{m17.1}
(i) A block consists of a finite number of bounded line segments with slope $-1$.

(ii)
The parameters $\rho, \theta_1$ and $\eps$ do not uniquely define the block $S$.
We will adopt the following convention. We fix a single block $S_{1, \theta_1, \eps}$ among all blocks with parameters $1, \theta_1, \eps$. Then we let
$S_{\rho, \theta_1, \eps} = \{x\in \R^2: x/\rho \in S_{1, \theta_1, \eps}\}$ for every set of parameters
$\rho, \theta_1, \eps$.

(iii) Since $ S_{\rho, \theta_1, \eps} \subset Q_{\eps,\rho}$ and $\eps<1$, the diameter of $ S_{\rho, \theta_1, \eps}$ is less than $3\sqrt{2} \rho$.

(iv)
For our arguments, it is irrelevant that the line segments
in the definition of a block are created in sets of three from each diamond in a diamond set.
We only need a finite family of line segments with slope $-1$ which satisfies estimates for the size of projections of a block given in
Lemma \ref{m16.2} below. A set of the type illustrated in Fig.~\ref{fig2}
\begin{figure}
\centering
 \includegraphics[width=5cm]{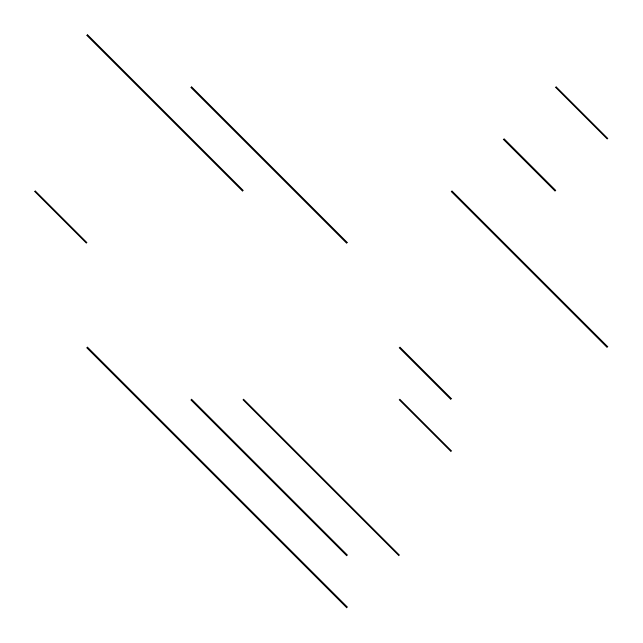}
\caption{A finite family of line segments with slope $-1$.}\label{fig2}
\end{figure}
is perfectly
acceptable as a ``building block'' if it satisfies the conditions listed in Lemma \ref{m16.2}.
\end{remark}

\begin{lemma}\label{m16.2}
Suppose that $\rho>0$, $\eps\in(0,1)$, $\theta_1\in(0,\arctan(1/3))$
and $S= S_{\rho,\theta_1, \eps}$ is a block. 
\begin{enumerate}[(i)]
\item There exists $\Theta_1 \subset [0, \theta_1]$ such
    that $\Leb([0,\theta_1] \setminus \Theta_1) \leq \eps$
    and for all $\theta \in \Theta_1$,
\begin{align*}
\Leb (\Pi_\theta S \om I_{\theta,\rho} ) \leq\eps\rho.
\end{align*}
\item There exists $\Theta_2 \subset ( \theta_1,\pi)$ such
    that $\Leb((\theta_1,\pi) \setminus \Theta_2) \leq
    \eps$ and for all $\theta \in \Theta_2$,
\begin{align*}
\Leb (\Pi_\theta S  ) \leq\eps\rho.
\end{align*}
\end{enumerate}
\end{lemma}

\begin{proof}
Let $A,G_k,C_k,\Theta_1$ and $\Theta_2$ be as in Corollary \ref{m16.1} and Definition \ref{m16.3}.
We have $\Leb([0,\theta_1] \setminus \Theta_1) \leq \eps$.
If $\theta\in(0,\arctan(1/3))$
then $\Pi_\theta C_k = \Pi_\theta G_k$. Hence,
by Corollary \ref{m16.1} (i), for all $\theta \in \Theta_1$,
\begin{align*}
\Leb (\Pi_\theta S \om I_{\theta,\rho} ) \leq\eps\rho.
\end{align*}

Since $C_k\subset G_k$, Corollary \ref{m16.1} implies that
$\Leb((\theta_1,\pi) \setminus \Theta_2) \leq \eps$ and
for all $\theta \in \Theta_2$,
\begin{align*}
\Leb (\Pi_\theta S  ) \leq\eps\rho.
\end{align*}
This completes the proof.
\end{proof}

\section{The sea urchin}\label{hedge}

Fix an arbitrarily small $\eps>0$. Our construction of the set
$F$ in Theorem \ref{m16.10} will have several parameters---real numbers
$r_1,\rho,\eps_1>0$ and an integer $N$. 
Assume that
\begin{align}
&\eps_1\rho N^2/(\pi r_1) < \eps/(16 \pi), \label{m18.6}\\
&N \eps_1 < \eps/(16N). \label{m18.12}
\end{align}
We will make more assumptions later in the proof.

Recall that $r_1>0$ is a (small) real number and let $N>0$ be a (large) integer divisible by 4. Let $a_k = r_1 \exp( i (k-1/2)2\pi/N)$ for $0\leq k \leq N $ and note that $a_N = a_0$. Let $M'_k$ be the closed rectangle with two of its adjacent vertices equal to $a_k$ and $a_{k+1}$, and the other two vertices on the unit circle $\prt D$. Moreover, we require that $M'_k$ does not contain $0$; this uniquely identifies $M'_k$. We let $M_k = M'_k \cup M'_{k+N/2}$ for $0\leq k \leq N/2 -1$. The set $M_k$ consists of two thin rectangles with parallel sides; their long sides lie on the same straight lines.
\begin{figure}
\centering
 \includegraphics[width=7cm]{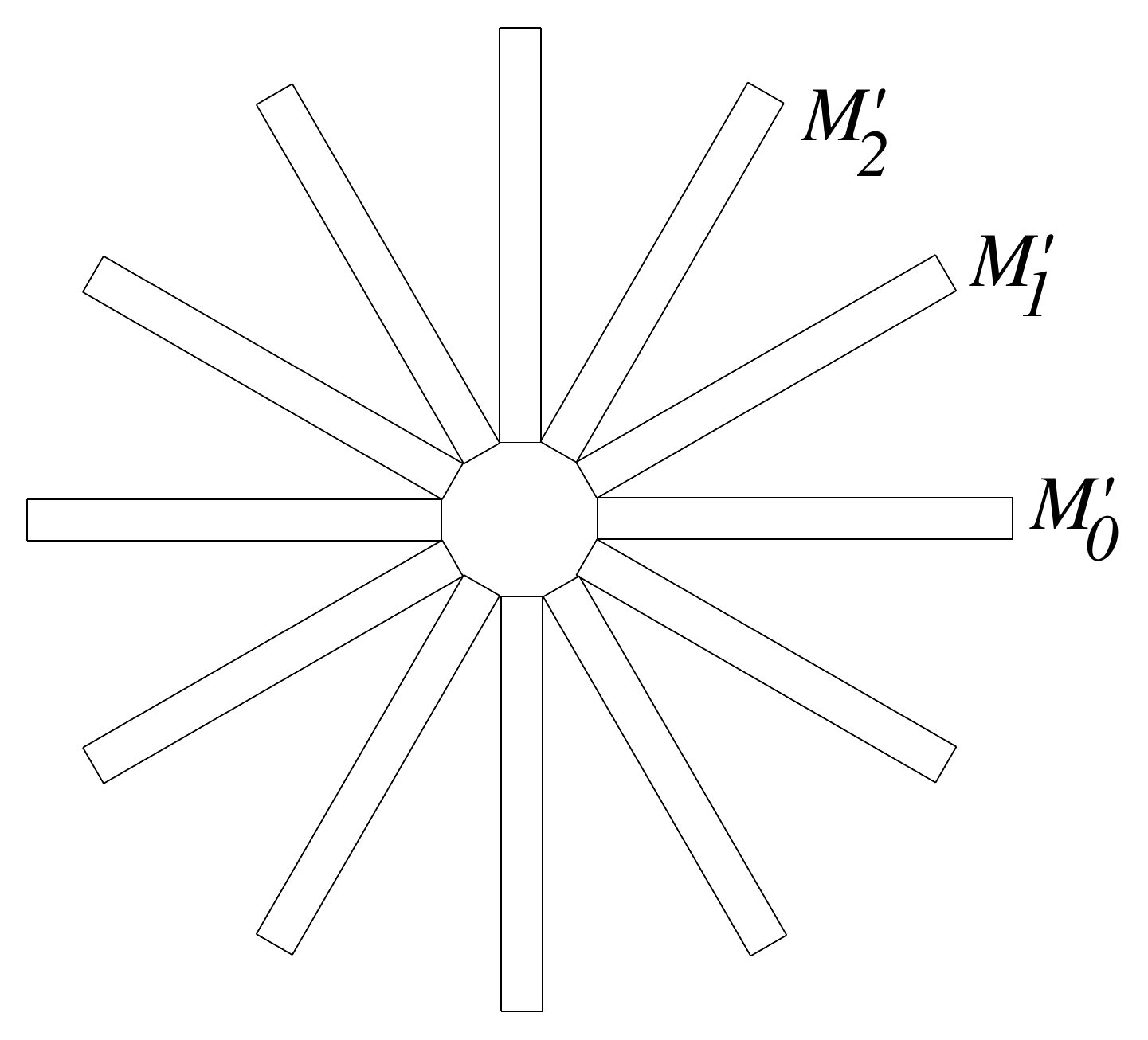}
\caption{Sets $M'_k$  are long rectangles which hold rows of blocks reflecting light rays.}\label{fig3}
\end{figure}
See Fig.~\ref{fig3}.

Recall definitions of squares $Q_{\eps_1,\rho}$ and blocks $S_{\rho, \theta_1, \eps_1}$ from Section \ref{blocks}.
We let $\theta_1 = 2\pi/N$ and assume that $N$ is so large that
$\theta_1\in(0,\arctan(1/3))$. Note that the width of
$M'_k$ is $|a_1-a_0|$ and the side of $Q_{\eps_1,\rho}$ is
$\rho(1+ 2\eps_1)$. We choose the values of the parameters so
that $q:= \rho(1+ 2\eps_1) = |a_1-a_0|$. Let $T_j: \R^2 \to
\R^2$ be the translation which maps $(-\rho \eps_1, -\rho\eps_1)$
to $a_0 + (j q, 0)$. The translations $T_j$, $j=0,1,2, \dots$, map
$Q_{\eps_1,\rho}$ onto adjacent squares which fill
up the rectangle $M'_0$ and then extend beyond $M'_0$. Let $j_*$ be the largest $j$ such that
$T_j(Q_{\eps_1,\rho}) \subset M'_0$. Then we let $F'_0 =
\bigcup_{0\leq j \leq j_*} T_j(S_{\rho, \theta_1, \eps_1})$,
$F'_k = e^{i2\pi k/N} F'_0$, $F_k = F'_k \cup F'_{k+N/2}$, and
$F = \bigcup_{0\leq k \leq N/2-1} F_k$.

Heuristically speaking, we placed blocks congruent to $S_{\rho, \theta_1, \eps_1}$ in long rows in $M_k$'s, tightly against each other. The  thin rows of blocks form a spiny sea urchin shape (see Fig.~\ref{fig3}).

We will prove that the set $F$ satisfies Theorem \ref{m16.10},
provided we choose appropriate values of the parameters of the
construction.

\subsection{The shadow}\label{shad}

\begin{proof}[Proof of Theorem \ref{m16.10} (i)]

We parametrize $\bv$ and $\bw$ using polar coordinates as follows, $\bv = e^{i(\theta+\pi/2)}$ and $\bw(r) = re^{i\theta}$; we will suppress $\theta$ in this notation.
We will estimate the measure of the set of $r\in(0,1)$ such that
$L_{\bv, \bw(r)}$ intersects $F_0$ for some values of $\theta$.

Let $I$ be the horizontal line segment extending from the vertical axis to the right hand side half of the boundary of the unit disc, at the level $-q/2$.  This line segment contains the lower horizontal side of $M'_0$.
Let $I_* = I \cap \bigcup_{0\leq j \leq j_*} T_j(Q_{\eps_1,\rho})$.
Suppose that $\theta\in [0,\theta_1]$.
It is elementary to see that we can choose
$N$ so large (and, hence, $\theta_1$ so small) that
$L_{\bv, \bw(r)} \cap I_* = \emptyset$ for $r\leq 0$. Moreover, we can choose
$r_1>0$ so small that for every $\theta\in [0,\theta_1]$,
\begin{align}\label{m19.1}
\frac
{\Leb(\{r>0: L_{\bv, \bw(r)} \cap I_* \ne \emptyset\})}
{\Leb(\{r>0: L_{\bv, \bw(r)} \cap D \ne \emptyset\})}
\geq 1- \eps/2^9.
\end{align}

Consider $\theta \in \Theta_1$, where $\Theta_1$ is as in Lemma
\ref{m16.2}, with $\eps$ replaced by $\eps_1$. Suppose that $L_{\bv, \bw(r)} \cap I_* \cap
T_0(Q_{\eps_1,\rho}) \ne \emptyset$. Then $L_{\bv, \bw(r)}$ may
fail to intersect $F_0$ if it crosses $T_0(Q_{\eps_1,\rho})$ to the left of
$T_0((0,0))$ or to the right of $T_0((\rho,0))$. The measure of the set of $r$ 
such that $\theta \in \Theta_1$, $L_{\bv, \bw(r)} \cap I_* \cap
T_0(Q_{\eps_1,\rho}) \ne \emptyset$, and $L_{\bv, \bw(r)}$ crosses $T_0(Q_{\eps_1,\rho})$ to the left of
$T_0((0,0))$ or to the right of $T_0((\rho,0))$
is bounded by
$ 2 \rho \eps_1 \cos \theta $. It follows from Lemma \ref{m16.2} (i) that the measure of the set of $r$ such that $L_{\bv, \bw(r)}$  crosses $
T_0(Q_{\eps_1,\rho})$ to the right of $T_0((0,0))$ and to the left of $T_0((\rho,0))$, and $L_{\bv,
\bw(r)}$ does not intersect $T_0(S_{\rho, \theta_1, \eps_1})$
is bounded by $\rho\eps_1$. Combining the two estimates, we
obtain
\begin{align*}
&\frac
{\Leb(\{r>0: L_{\bv, \bw(r)} \cap I_* \cap T_0(Q_{\eps_1,\rho})\ne \emptyset,
L_{\bv, \bw(r)} \cap F_0 \ne \emptyset
\})}
{\Leb(\{r>0: L_{\bv, \bw(r)} \cap I_* \cap T_0(Q_{\eps_1,\rho})\ne \emptyset\})}
\\
&\geq \frac{ \rho(1+2\eps_1) \cos \theta - 2\rho \eps_1 \cos \theta -\rho\eps_1}
{ \rho(1+2\eps_1) \cos \theta}\\
&= \frac {1-\eps_1/\cos \theta} {1+2\eps_1}\\
&\geq \frac {1-\eps_1/\cos \theta_1} {1+2\eps_1}.
\end{align*}
We may assume that $\theta_1$ is so small that $\cos \theta_1 \ge 1/2$. Hence the last expression is bounded from below by $(1- 2 \eps_1)/(1+2\eps_1)$. 

The same estimate holds for $T_j$ in place of $T_0$ for $j=1,\dots, j_*$, by translation invariance. Since
$I_* = I \cap \bigcup_{0\leq j \leq j_*} T_j(Q_{\eps_1,\rho})$, this implies that
\begin{align*}
\frac
{\Leb(\{r>0: L_{\bv, \bw(r)} \cap I_* \ne \emptyset,
L_{\bv, \bw(r)} \cap F_0 \ne \emptyset
\})}
{\Leb(\{r>0: L_{\bv, \bw(r)} \cap I_* \ne \emptyset\})}
&\geq \frac {1-2\eps_1} {1+2\eps_1}.
\end{align*}
We can now make $\eps_1>0$ so small that the last estimate and \eqref{m19.1} yield
\begin{align*}
\frac
{\Leb(\{r>0:
L_{\bv, \bw(r)} \cap F_0 \ne \emptyset
\})}
{\Leb(\{r>0: L_{\bv, \bw(r)} \cap D \ne \emptyset\})}
&\geq 1-\eps/2^8.
\end{align*}

Recall that $\bv = e^{i(\theta+\pi/2)}$, $\bw(r) =
re^{i\theta}$ and $\Theta_1 \subset [0,\theta_1]$. The last
inequality implies that the $\mtwo$-measure of $(\bv,\bw)$ such
that $\theta \in \Theta_1$, $r\in(0,1)$ and $L_{\bv, \bw(r)}
\cap F_0 = \emptyset$ is bounded above by $\theta_1 \eps/2^8$.
According to Lemma \ref{m16.2} (i), $\Leb([0,\theta_1]
\setminus \Theta_1) \leq \eps_1$. It follows that the
$\mtwo$-measure of $(\bv,\bw)$ such that $\theta \in
[0,\theta_1]$, $r\in(0,1)$ and $L_{\bv, \bw(r)} \cap F_0 =
\emptyset$ is bounded above by $\theta_1 \eps/2^8 + \eps_1$.
Summing over all intervals of the form $[k\theta_1, (k+1)
\theta_1]$ and taking into account both $\bv =
re^{i(\theta+\pi/2)}$ and $\bv = re^{i(\theta-\pi/2)}$, we
obtain the following estimate,
\begin{align*}
\mtwo(\{(\bv, \bw)\in \calV:
L_{\bv, \bw} \cap F = \emptyset\})
&\leq 2 N (\theta_1 \eps/2^8 + \eps_1)
= 2 N ((2\pi/N)\eps/2^8 + \eps_1)\\
&\leq \eps/16 + 2N\eps_1.
\end{align*}
We can now make $\eps_1$ so small that the right hand side is less than $ \eps/8 $, i.e.,
\begin{align}\label{m24.1}
\mtwo(\{(\bv, \bw)\in \calV:
L_{\bv, \bw} \cap F = \emptyset\})
< \eps/8.
\end{align}

We have constructed a set $F$ satisfying part (i) of Theorem
\ref{m16.10}. Of course, part (i) is a trivial statement by
itself. We will have to show that the same set $F$ satisfies
part (ii) of the theorem.
\end{proof}

\section{Light ray reflections}\label{reflect}

\begin{lemma}\label{m18.14}
Recall that $\eps >0$ is fixed. There exists a set $\calA_1
\subset \calV$ such that $\mtwo(\calV \setminus \calA_1) \leq
\eps/8 $ and if $(\bv, \bw) \in \calA_1$, $\bw = r e^{i
\theta}$, and $k$ satisfies $\theta - 2\pi k /N \in [0,
\theta_1]\cup [\pi,\pi+\theta_1]$ then
\begin{align}\label{m20.2}
\Leb (\Pi_\theta (F\setminus  M_k)) \leq  \eps/(16 \pi) .
\end{align}
\end{lemma}

\begin{proof}
Let $\Theta_2$ be as in Lemma \ref{m16.2} (ii), with $\eps$ replaced by $\eps_1$. Let
\begin{align*}
\Theta'_2 &= \{t \in \R: t + m \pi \in \Theta_2
\text {  for some  } m \in \Z\},\\
 \Theta^k_2 &=  \Theta'_2 + 2\pi k/N,\\
\Theta^*_2 &= \bigcap_{1\leq k \leq N/2-1}
 \Theta ^k_2.
\end{align*}
Then \begin{align}\label{m20.1} \Leb( ([0, \theta_1]\cup
[\pi,\pi+\theta_1]) \setminus \Theta^*_2) \leq N \eps_1 <
\eps/(16N),
\end{align}
by Lemma \ref{m16.2} (ii) and \eqref{m18.12}. Let $\calA_1$ be
the set of all $(\bv,\bw)= (e^{i(\theta+\pi/2)}, \bw) \in
\calV$ such that if  $k$ satisfies  $\theta -2\pi k/N \in [0, \theta_1 ] \cup
[ \pi, \pi  + \theta_1]$ then $\theta \in \Theta^{*}_2 + 2\pi k/N$. It follows from Definition \ref{m18.5} of $\mtwo$ and
\eqref{m20.1} that $\mtwo(\calV \setminus \calA_1) \leq \eps/8$. It remains to prove \eqref{m20.2}.

Recall the integer $j_*$ used in the construction of $F$. It is
elementary to check that if $N$ is large then $j_* \leq N/( \pi
r_1) $. Hence, the number of block images $T_j(S_{\rho,
\theta_1, \eps_1})$ inside $M'_0$ is bounded by $N/( \pi r_1) $.
Therefore, the number of (rotated) block images in $F$ is
bounded by $N^2/( \pi r_1) $. According to Lemma \ref{m16.2} (ii)
we have $\Leb (\Pi_\theta T_j(S_{\rho, \theta_1, \eps_1})  )
\leq\eps_1\rho$ for $\theta \in \Theta_2$. Hence, we obtain
using \eqref{m18.6},
\begin{align}\label{m20.3}
\Leb (\Pi_\theta (F\setminus  M_0)  ) \leq\eps_1\rho N^2/( \pi r_1) < \eps/(16 \pi),
\end{align}
for $\theta \in \Theta^*_2$.

If $(\bv,\bw)= (e^{i(\theta+\pi/2)}, \bw) \in\calA_1 $ and $k$ satisfies $\theta -2\pi k/N \in [0, \theta_1 ] \cup [ \pi,
\pi + \theta_1]$ then $\theta \in \Theta^{*}_2 + 2\pi k/N$.
For such $\theta$, we obtain from \eqref{m20.3}, by rotation invariance,
\begin{align*}
\Leb (\Pi_\theta (F\setminus  M_k)  )  < \eps/(16 \pi) .
\end{align*}
\end{proof}

We will now describe the path of a light ray reflecting from
mirrors in $F$ in general terms. Let $\calA_1$ be as in Lemma
\ref{m18.14}. Suppose that the light ray arrives along the line
$\Gamma_{\bv,\bw}$ with $(\bv, \bw) \in \calA_1$, $\bv = e^{i
(\theta+\pi/2)}$, and $\theta  \in [0, \theta_1]$. According to
Lemma \ref{m18.14}, this ray is very unlikely to hit $F$ before
hitting $M_0$. So let us suppose that it did not hit $F$ before
hitting $M_0$.

The set $F_0$ consists of a finite number of line segments with
slope $-1$. The light ray may reflect from a number of them.
After each reflection, it will move along a line
$L_{\bv_1,\bw_1}$ with $\bv_1$ equal either to $ e^{i
(\theta+\pi/2)}$ or $ e^{i (\pi-\theta)}$. It is clear that
after a finite number of reflections, the light ray will leave
the set $M_0$. We will argue that at the time the light ray
leaves $M_0$, it is very likely to move along a line
$L_{\bv_2,\bw_2}$ with $\bv_2= e^{i (\theta+\pi/2)}=\bv$.

Next, the light ray will have another chance to reflect from
$F\setminus M_0$. We will show that the chance that the light
ray will hit $F\setminus M_0$ is very small, once again using Lemma
\ref{m18.14}.

In summary, a typical light ray arriving in the direction $\bv
= e^{i (\theta+\pi/2)}$ with $\theta  \in [0, \theta_1]$ will
avoid hitting $F\setminus M_0$ on the way to $M_0$, then it
will follow a zigzag path inside $M_0$, and then it will leave
the unit disc without hitting $F\setminus M_0$ on the way out.
A similar analysis applies to light rays arriving from other
directions.

\subsection{Invariance principle for light rays}

By abuse of language, we will refer to $\calA\subset\calV$ as a bundle of light rays, although it would be more precise to say the a bundle of light rays consists of all light rays $\Gamma_{\bv,\bw}$ with $(\bv,\bw)\in \calA$.

\begin{lemma}\label{m20.4}
Recall that $K_\theta = \{z=re^{i\theta}: r\in \R\}$ and
suppose that $B_1 \subset K_\theta$ and $\Leb(B_1)= b_1$.
Consider a bundle $\calB_1 = \{ (\bv, \bw): \bw \in B_1, \bv =
e^{i (\theta+\pi/2)} \}$ of parallel light rays. Suppose that
all light rays in $\calB_1$ reflect from a bounded set $C$
consisting of a finite number of parallel line segments and
then escape as two bundles of light rays $\calB_2 = \{
(\bv, \bw): \bw \in B_2, \bv = e^{i (\theta+\pi/2)} \}$ and
$\calB_3 = \{ (\bv, \bw): \bw \in B_3, \bv = e^{i
(\alpha+\pi/2)} \}$, with $\alpha \in [0,2\pi)$, $B_2 \subset K_\theta$, $\Leb(B_2)=
b_2$, $B_3 \subset K_\alpha$, and $\Leb(B_3)= b_3$. Let $\calH$
be the transformation that takes an element of $\calB_1$ and
maps it to the outgoing light ray in $\calB_2 \cup \calB_3$.
The transformation $\calH$ is one-to-one, except for a finite
number of lines in $\calB_1$ for which $\calH$ is not uniquely
defined because these light rays encounter endpoints of mirrors
on their way. 

(i) We have $b_2 + b_3 = b_1$. 

(ii) Moreover, $\calH^{-1}
(\calB_2) = \calB'$ and $\calH^{-1} (\calB_3) = \calB''$, where
$\calB' = \{ (\bv, \bw): \bw \in B', \bv = e^{i
(\theta+\pi/2)} \}$, $\Leb(B')= b_2$, $\calB'' = \{ (\bv, \bw):
\bw \in B'', \bv = e^{i (\theta+\pi/2)} \}$, and $\Leb(B'')=
b_3$.
\end{lemma}

\begin{proof}

The claim is obvious if $C$ consists of a single line segment. The general statement can be easily proved by induction on the number of line segments in $C$.
We leave the details to the reader.
\end{proof}

We note parenthetically that Lemma \ref{m20.4} is a special case of a well known and more general theorem in the theory of billiards, see \cite[Thm. 3.1]{T} or \cite[Lemma 2.35]{CM}.

\bigskip

Consider $\theta \in [0, \theta_1]$ and the bundle $\calB_1$ of
light rays $(\bv, \bw)$ such that $\bv = e^{i (\theta+\pi/2)}$
and $L_{\bv, \bw} \cap M_0 \ne \emptyset$. We can write
$\calB_1 = \{ (\bv, \bw): \bw \in B_1, \bv = e^{i
(\theta+\pi/2)} \}$ for some $B_1 \subset K_\theta$. Let $b_1 =
\Leb(B_1)$. In the following lemma, we will ignore the set $F\setminus F_0$, that is, we will consider
the effect of reflections in $F_0$ on the light rays in
$\calB_1$. After the light rays leave the set $M_0$, they will
form  two bundles of parallel light rays $\calB_2 = \{ (\bv,
\bw): \bw \in B_2, \bv = e^{i (\theta+\pi/2)} \}$ and $\calB_3
= \{ (\bv, \bw): \bw \in B_3, \bv = e^{i (\pi - \theta)} \}$,
where $B_2 \subset K_\theta$, $\Leb(B_2)= b_2$, $B_3 \subset
K_{\pi/2 - \theta}$, and $\Leb(B_3)= b_3$.

\begin{lemma}\label{m20.5}

Suppose that $\eps>0$ and $\theta \in [0, \theta_1]$ are fixed.

(i) We can make $r_1>0$ so small and $N$ so large that
$b_3<  \eps/(32 \pi) $.

(ii) Let $B_4 \subset B_2$ be the set of all $\bw_1$ such that $\Gamma_{\bv, \bw_1}$ with
$(\bv, \bw_1) \in \calB_2$ is the escape trajectory for some
light ray arriving along $\Gamma_{\bv, \bw}$ with $(\bv, \bw)\in \calB_1$ and $|\bw -
\bw_1| \geq \eps$. Let $b_4 = \Leb(B_4)$. We can make $r_1>0$
so small and $N$ so large that $b_4 <  \eps/(32 \pi) $.
\end{lemma}

\begin{proof}
(i) Note that if $(\bv, \bw) \in \calB_3 $ then $\bw \in B_3
\subset K_{\pi/2 - \theta}$. We have $B_3 \subset \Pi_{\pi/2 -
\theta} M_0$. It is easy to see that one can make $r_1>0$ so
small and $N$ so large that $\Leb(\Pi_{\pi/2 - \theta} M_0)<
 \eps/(32 \pi) $ for all $\theta \in [0, \theta_1]$. It follows that
$b_3 = \Leb(B_3) \leq \Leb(\Pi_{\pi/2 - \theta} M_0)<  \eps/(32 \pi) $
for all $\theta \in [0, \theta_1]$.

(ii) Recall that $q$ denotes the width of $M_0'$. Let $z^0=
(z^0_1, z^0_2)$ be the lower right corner of $M'_0$ and  $z^k =
(z^0_1 - k \eps/2, z^0_2)$ for $k\geq 0$. Let $\calB_1^k$ be
the bundle of light rays in $\calB_1$ which enter $M'_0$
through the line segment $\ol {z^k, z^{k+1}}$. Let $D_k$ be the
(unique) open rectangle with two sides on the lines that
contain the longs sides of $M_0'$ and such that two of its
corners are $z^k$ and $z^{k+1}$. Let $k_*$ be the maximum $k$
such that $D_k \subset M'_0$ and note that $k_* \leq 2/\eps $.
Let $\wh I_k$ be the upper side of $D_k$ and let $\wt I_k$ be
the left side of $D_k$. Note that light rays in $\calB_1^k$
enter $D_k$ at the same time when they enter $M'_0$. After
reflecting in $F_0 \cap D_k$, they have to exit $D_k$ either
through $\wh I_k$ or $\wt I_k$. In the following definition,
$L_{\bv, \bw}^*$ represents the part of the set  $L_{\bv,
\bw}$ lying outside $D_k$. 
Since $D_k$ is open, we will use the statement $L_{\bv, \bw}^* \cap \wh I_k \ne \emptyset$ in the definitions below to indicate that $\Gamma_{\bv, \bw} $ exits $D_k$ through $ \wh I_k$ (the same remark applies to similar statements).  
Light rays exiting $D_k$ can be
grouped into the following four bundles,
\begin{align*}
\wh \calB^k_2 &= \{ (\bv, \bw): \bw \in \wh B^k_2, \bv = e^{i (\theta+\pi/2)}, L_{\bv, \bw}^* \cap \wh I_k \ne \emptyset \}, \\
\wt \calB^k_2 &= \{ (\bv, \bw): \bw \in \wt B_2^k, \bv = e^{i (\theta+\pi/2)}, L_{\bv, \bw}^* \cap \wt I_k \ne \emptyset \}, \\
\wh \calB^k_3 &= \{ (\bv, \bw): \bw \in \wh B_3^k, \bv = e^{i (\pi - \theta)}, L_{\bv, \bw}^* \cap \wh I_k \ne \emptyset \},\\
\wt \calB^k_3 &= \{ (\bv, \bw): \bw \in \wh B_3^k, \bv = e^{i (\pi - \theta)}, L_{\bv, \bw}^* \cap \wt I_k \ne \emptyset \}.
\end{align*}
We have
\begin{align*}
\wh B^k_2 \cup \wt B^k_2 \subset K_\theta, \qquad
\wh B^k_3 \cup \wt B^k_3 \subset K_{\pi/2 - \theta},
\end{align*}
and we let
\begin{align*}
\wh b^k_2 = \Leb(\wh B^k_2), \qquad
\wt b^k_2 = \Leb(\wt B^k_2), \qquad
\wh b^k_3 = \Leb(\wh B^k_3), \qquad
\wt b^k_3 = \Leb(\wt B^k_3).
\end{align*}
Note that $\Leb(\wt I_k)= q$ so $\Leb(\Pi_\theta \wt I_k) +
\Leb(\Pi_{\pi/2 - \theta} \wt I_k) \leq  2 q$.
Since $\wt B^k_2 \subset \Pi_\theta \wt I_k$ and $\wt B^k_3
\subset \Pi_{\pi/2 - \theta} \wt I_k$, we obtain $\Leb(\wt
B^k_2) + \Leb(\wt B^k_3)  < 2 q$.
This implies that 
\begin{align*}
\sum_{0\leq k \leq k_*}(\Leb(\wt B^k_2) + \Leb(\wt
B^k_3))  < 2 q(1+k_*) \leq 2 q (1 + 2/\eps).
\end{align*}
We can make
$r_1>0$ so small and $N$ so large that $2 q (1 + 2/\eps) <  \eps/(64 \pi)  $. Then  
\begin{align}\label{m25.5}
\sum_{0\leq k \leq k_*}(\Leb(\wt B^k_2) + \Leb(\wt
B^k_3))  <  \eps/(64 \pi)  .
\end{align}

If a light ray enters one of the
rectangles $D_k$ and it does not exit $D_k$ through $\wt I_k$
then the distance between the entry point and the exit point
from $D_k$ is less than $\eps$.

Let $\{ (\bv, \bw): \bw \in B_5, \bv = e^{i
(\theta+\pi/2)} \} \subset \calB_1$ represent all light rays that
hit $M'_0$ at a point which does not belong to $\bigcup_{0\leq
k \leq k_*} \ol D_k$. It is easy to see that we can make
$r_1>0$ so small and $N$ so large that 
\begin{align}\label{m25.6}
\Leb(B_5) <  \eps/(64 \pi)  .
\end{align}

If a light ray arrives along $\Gamma_{\bv, \bw} $ with $(\bv, \bw) \in \calB_1$ and escapes along $\Gamma_{\bv, \bw_1} $ with $\bw_1 \in B_4$ then 
\begin{enumerate}[(i)]
\item 
either $\bw \in B_5$,
\item
or $\Gamma_{\bv,\bw} $ enters a set $\ol D_k$ at the
same time when it enters $M'_0$ and 
exits $D_k$ along a line $\Gamma_{\bv_2,\bw_2} $ 
with $(\bv_2,\bw_2) \in \wt \calB^k_2 \cup \wt \calB^k_3$.
\end{enumerate} 
This, Lemma \ref{m20.4} and \eqref{m25.5}-\eqref{m25.6} imply that $\Leb(B_4) <  \eps/(32 \pi) $.
\end{proof}

\begin{proof}[Proof of Theorem \ref{m16.10}]

We will define several families of lines representing the
progress of light rays on their way through the disc $D$. The
bundles will represent only those light rays that have
``desirable'' trajectories. The bundles of light rays will have
the form $\calB_j = \{ (\bv, \bw): \bw \in B^\theta_j, \bv =
e^{i (\theta+\pi/2)}, \theta \in [0, \theta_1] \}$ with
$B_j^\theta \subset K_\theta$. Estimates for light rays with $\theta \notin [0, \theta_1]$ can be obtained by rotation invariance.

The first bundle $\calB_6$ represents all light rays that hit
the unit disc and are perpendicular to a line $K_\theta$ with
$\theta \in [0, \theta_1]$. In other words, $B^\theta_6 =
\{x\in  K_\theta: |x| < 1\}$ and $\Leb(B^\theta_6)= 2$.

Let $\calB_7$ be the bundle of those light rays in $ \calB_6$
which hit $F$.

Let $\calB_8$ be the bundle of those light rays in $ \calB_7 $
which hit $M_0$ but do not hit $F$ before hitting $M_0$. 

Let $\calB_9$ be the family of all lines $L_{\bv_1, \bw_1}$
which represent light rays at the exit time from $M_0$,
assuming that they satisfy the following three conditions. First, the light rays entered $M_0$ along a line $(\bv, \bw)\in
\calB_8$.
Second, $\bv_1 = \bv$;  in other words, the exiting light ray moves
along a line parallel to the one along which it was moving at
the hitting time of $M_0$. Thirdly,
$|\bw - \bw_1| \leq \eps$.

 Let $\calB_{10}$ represent those lines in $\calB_9$
which do not hit $F \setminus M_0 $. Heuristically, it would be more natural
to define $\calB_{10}$ as the family of those lines in
$\calB_9$ which represent light rays which do not hit $F$ \emph{after exiting $M_0$}.
However, we believe that our definition of $\calB_{10}$ makes
our argument a little bit easier to understand.

Finally, let $\calB_{11}$ be the family of all lines $L_{\bv, \bw}$ in $\calB_{8}$ representing light rays which, after reflecting from mirrors in $F_0$, exit $M_0$ along a line $L_{\bv_1, \bw_1}\in \calB_{10}$.

Note that $\mtwo(\calB_6) = \mtwo(\calV)/N = 4\pi/N $ because $\theta_1 = 2\pi/N$.   A similar argument and the fact that the set
$F$ is invariant under rotations by angles $k\theta_1$,
$k\in \Z$, apply to other
estimates, for example, \eqref{m24.1} implies
that $\mtwo(\calB_6 \setminus \calB_7) < \eps/(8N) $.

The last estimate and Lemma \ref{m18.14} imply that $\mtwo(\calB_7 \setminus \calB_8) < \eps/(8N) + (\eps/(16 \pi))2 \pi/N = 2 \eps/(8N)$. This gives $\mtwo(\calB_6 \setminus \calB_8) < 3\eps/(8N) $. 
We combine this estimate,  Lemma \ref{m20.4} and Lemma \ref{m20.5} to derive the following inequality, $\mtwo(\calB_6 \setminus \calB_9) <  3\eps/(8N) + 2 (\eps/(32\pi))2\pi/N = 4 \eps/(8N)$.
At this point, we apply Lemma \ref{m18.14} again to derive the following estimate, $\mtwo(\calB_6 \setminus \calB_{10}) <  6\eps/(8N)$.

Using Lemma \ref{m20.4} we obtain $\mtwo(\calB_6 \setminus \calB_{11}) = \mtwo(\calB_6 \setminus \calB_{10}) < 6\eps/(8N)$.

Let $\calA = \bigcup_{0\leq k \leq N-1} \{(\bv, \bw):  (e^{2\pi ik/N} \bv,e^{2\pi ik/N} \bw) \in \calB_6 \setminus \calB_{11} \}$. 
Then $\mtwo(\calV \setminus\calA) \leq \eps$, by rotation invariance of $F$. It easy to check that $\calA$ satisfies the conditions stated in parts (i) and (ii) of the theorem.
\end{proof}

\section{Invisible mirror}\label{invmir}

This section contains the proof of Theorem \ref{a3.1}.

Let $\calS_n$ denote the family of all permutations of integers $1,2,\dots,n$. 
Let $Z$ be the rhombus (a convex closed quadrilateral) with sides on the lines $K_0$, $K_0+(0,1)$, $K_\theta$ and $K_\theta + e^{i(\theta+\pi/2)}$. Let $L$ be the diagonal of $Z$ which does not contain $(0,0)$.
For $\bs =(j_1,j_2, \dots , j_n)\in \calS_n$ and a set $A \subset \R^2$, let 
\begin{align*}
\calM_{n,\bs,k} (A)&=  
(1/n)A -\frac{k-1}n  \frac{e^{0\cdot i}}{\sin \theta} + \frac{j_k-1}n  \frac{e^{i \theta}}{\sin \theta}, \qquad 1\leq k \leq n,\\
\calM_{n,\bs} (A)&=  
\bigcup_{1\leq k \leq n} \calM_{n,\bs,k} (A),\\  
Z_{n,\bs} &=\calM_{n,\bs} (Z).
\end{align*}

\begin{remark}\label{a3.6}
(i)
It is elementary to see that for any $n\geq 1$ and any permutation $\bs$,
the set $G = \calM_{n,\bs} (L)$ satisfies part (i) of Theorem \ref{a3.1} (see Fig.~\ref{fig4}).
It remains only to prove that for some $n$ and $\bs$, $G = \calM_{n,\bs} (L)$
satisfies part (ii) of the theorem.
\begin{figure}
\centering
 \includegraphics[width=10cm]{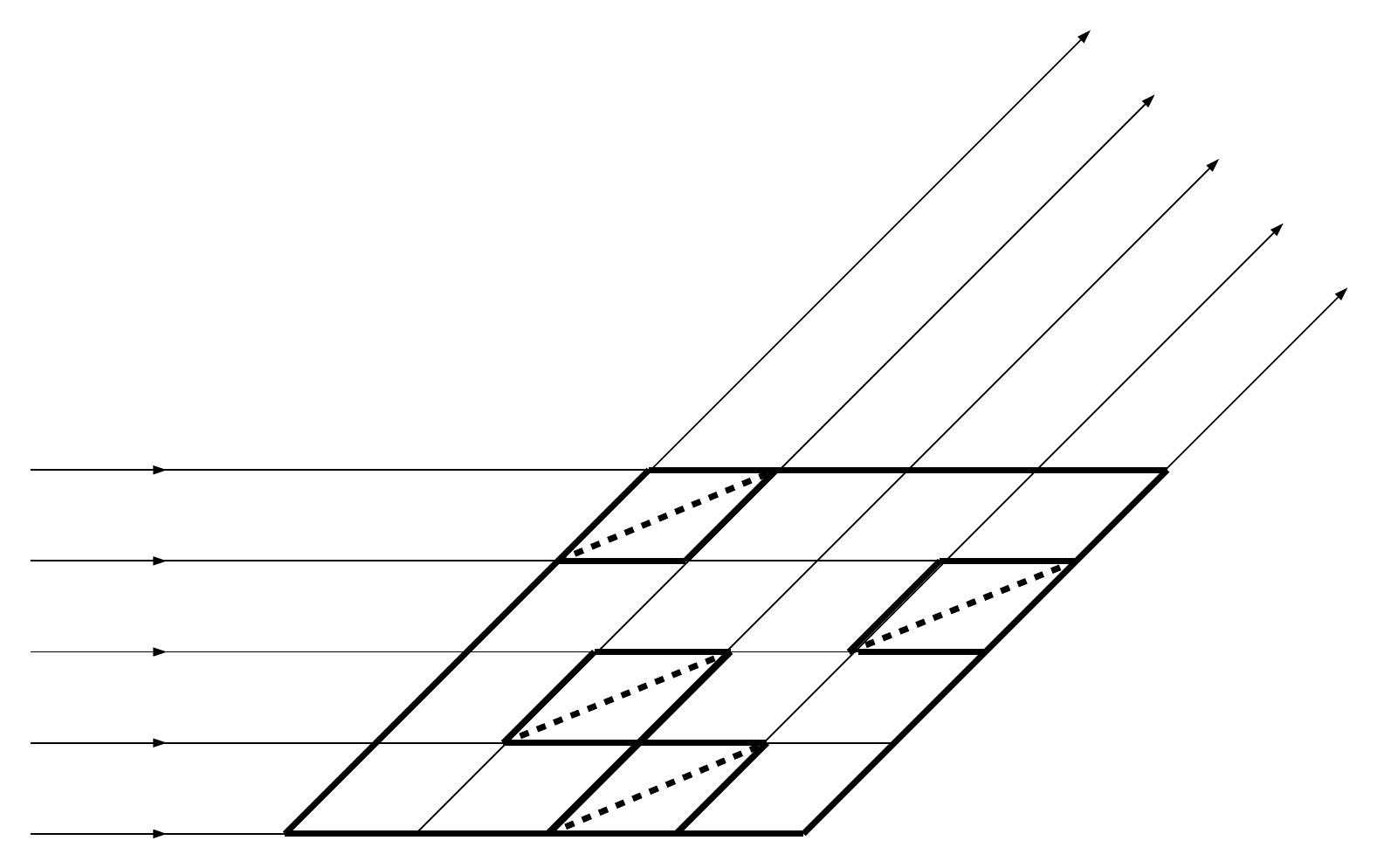}
\caption{A unit width bundle of light is coming from the left.
Exiting light rays form a
bundle of unit width inclined at an angle $\theta$ (in our picture, $\theta = \pi/4$). The rhombus $Z$ and the four
rhombuses comprising $Z_{4, \bs}$ with $\bs=  (3,1,2,4)$ are outlined
in bold. The mirrors comprising $\calM_{4,\bs} (L)$ are represented by dotted lines.
}\label{fig4}
\end{figure}

(ii) It is easy to see that $\calM_{n_1,\bs_1} (Z_{n_2,\bs_2})
= \calM_{n_1,\bs_1} (\calM_{n_2,\bs_2}(Z))$ is equal to $Z_{n_3,\bs_3}$
with $n_3 = n_1 n_2$ and some $\bs_3 \in \calS_{n_3}$. By induction,
\begin{align*}
\calM_{n_1,\bs_1} \circ\calM_{n_2,\bs_2}\circ \cdots
\circ\calM_{n_k,\bs_k}(Z)
\end{align*}
can be represented as $Z_{n_{*},\bs_{*}}$
with $n_* = n_1 n_2\cdots n_k$ and some $\bs_* \in \calS_{n_*}$.

(iii) Suppose that for some $\alpha$, $\delta$, $n$ and $\bs$ we have 
$\Leb(\Pi_\alpha Z_{n,\bs}) < \delta$. Then for any $n_1$ and $\bs_1$ we have
$\Leb(\Pi_\alpha \calM_{n_1,\bs_1}(Z_{n,\bs})) < \delta$. To see this,
note that for every $1\leq k \leq n_1$,
\begin{align*}
\Leb(\Pi_\alpha \calM_{n_1,\bs_1,k}(Z_{n,\bs}))
= (1/n_1) \Leb(\Pi_\alpha Z_{n,\bs}) < \delta/n_1,
\end{align*}and $\calM_{n_1,\bs_1}(Z_{n,\bs})$ is the union
of $n_1$ sets of the form $\calM_{n_1,\bs_1,k}(Z_{n,\bs})$.
We also trivially have 
$\Leb(\Pi_\alpha \calM_{n,\bs}(Z_{n_1,\bs_1})) \leq \Leb(\Pi_\alpha \calM_{n,\bs}(Z)) < \delta$. 
By induction, if $\Leb(\Pi_\alpha \calM_{n_j,\bs_j}(Z)) < \delta$
for some $1\leq j \leq k$ then
\begin{align}\label{a3.7}
\Leb(\Pi_\alpha 
(\calM_{n_1,\bs_1} \circ\calM_{n_2,\bs_2}\circ \cdots
\circ\calM_{n_k,\bs_k}(Z) )) < \delta.
\end{align}

\end{remark}

\bigskip

For a given $\alpha$, we will find $n $ and $\bs$ such that $\Leb(\Pi_\alpha Z_{n,\bs})$ is small.
Note that the vertices of rhombuses $\calM_{n,\bs,k}(Z)$ belong to the lattice $\calL_{n,\theta}$ generated by $ 1/(n \sin \theta) $ and $ e^{i\theta}/ (n \sin \theta) $, that is, the set of all points of the form $ k_1/(n \sin \theta) + e^{i\theta}k_2/(n \sin \theta) $, for integer $k_1$ and $k_2$.
We will first consider $\alpha$ such that $\alpha- \pi/2 \ne 0,\theta$ and $K_{\alpha-\pi/2}$
passes through a point $z= -k_1/\sin \theta + k_2 e^{i\theta}/\sin \theta \in \calL_{1,\theta}$, $z\ne (0,0)$. Such angles $\alpha$ can be considered to be ``rational'' in the context of our construction.
We will assume that $k_1$ and $k_2$ are positive. Other cases can be dealt with in a similar manner.

Consider a (large) integer $n_1>2k_1 k_2$ whose value will be specified later. 
We will call a family $\calC$ of pairs of integers a \emph{chain} if for every $(j_1,j_2), (j_3, j_4) \in \calC$ we have $1\leq j_1, j_2, j_3, j_4 \leq n_1$,  $j_1 - j_3 = m k_1$ and $ j_2 - j_4 = m k_2$ for some integer $m$. We will say that chains $\calC_1$ and $\calC_2$ are orthogonal and write $\calC_1 \perp \calC_2$ if 
for every $(j_1,j_2) \in \calC_1$ and $ (j_3, j_4) \in \calC_2$, $j_1\ne j_3$ and $j_2 \ne j_4$. 
We will say that a chain $\calC$ is \emph{maximal} if for every $(j_1,j_2) \in \calC$ and $1\leq  j_3, j_4 \leq n_1$ such that  $j_3 - j_1 = m k_1$ and $ j_4 - j_2 = m k_2$ for some integer $m \geq 1$ we have $(j_3,j_4) \in \calC$. Note that maximality is not symmetric; the maximal chain extends as far as possible in one direction but not necessarily in the opposite direction. 
We will always tacitly assume that the coordinates of pairs $(j,k)$ are in the range from 1 to $n_1$.
We will say that $(j_1,j_2)$ is the \emph{root} of $ \calC$ if for every $ (j_3, j_4) \in \calC$ we have $j_3 \geq j_1$ (note that then $j_4 \geq j_2$ because $k_1,k_2 >0$). We will also say that $\calC$ starts at $(j_1,j_2)$.

Next we will define a family of chains such that every pair of these chains is orthogonal. We will use induction.
Let $\calC_1$ be the  maximal  chain starting at $(1,1)$.

Suppose that $\calC_1, \calC_2, \dots, \calC_m$ have been chosen and suppose further that every pair of these chains are orthogonal. We will write $\bs(j_1) = j_2$ if  
$(j_1,j_2) \in \calC_r$ for any $1\leq r \leq m$. The formula defines a function $\bs$ on a subset of $\{1,2,\dots, n_1\}$ because the chains are orthogonal. For the same reason, the function $\bs$ is injective. If $\bs(j)$ is defined for all $1\leq j \leq n_1$ then we stop. 

Note that functions $\bs$ defined at different steps of the induction argument necessarily agree on the common parts of their domains.
For this reason, we did not put a subscript on $\bs$.

Suppose that $\bs(j_1)$ is not defined for some $1\leq j_1 \leq n_1$.  Let $r_1$ be the smallest of the integers $1 \le j_1 \le n_1$ such that $\bs(j_1)$ is not defined. Since $\bs$ is bijective on its domain, there must exist $1\leq j_2 \leq n_1$ which is not in the range of $\bs$.  Let $r_2$ be the smallest of the integers $1 \le j_2 \le n_1$ such that $j_2$ is not in the range of $\bs$. We will say that $(r_1,r_2)$ is the minimal root in the complement of $\calC_1 \cup \dots \cup \calC_m$.  Let $\calC_{m+1}$ be the maximal chain that starts at $(r_1,r_2)$.
We will argue that $\calC_{m+1}\perp \calC_r$ for all $1\leq r \leq m$.

Suppose that for some $1\leq r \leq m$, it is not true that $\calC_{m+1}\perp \calC_r$. Then there must exist $(j_1,j_2) \in \calC_{m+1}$ and $(j_3,j_4) \in \calC_r$ such that $j_1 = j_3$ or $j_2 = j_4$. We will assume that $j_1 = j_3$, the other case being analogous. We can assume without loss of generality that $j_1$ is the smallest of all integers with the property that $(j_1,j_2) \in \calC_{m+1}$ and $(j_1,j_4) \in \calC_r$ for some $j_2$ and $j_4$.
We will argue that $(j_1,j_4)$ is the root of $\calC_r$. Note that 
$(j_1,j_2) $ is not the root of $ \calC_{m+1}$ because the first coordinate of the root cannot agree with the first coordinate of any element of $\calC_1 \cup \dots \cup \calC_m$. Hence, $(j_1 - k_1 , j_2 - k_2)  \in \calC_{m+1}$. Suppose that $(j_1,j_4)$ is not the root of $\calC_r$. Then $(j_1 - k_1 , j_4 - k_2)  \in \calC_r$. The first coordinates of  $(j_1 - k_1 , j_2 - k_2)$ and  $(j_1 - k_1 , j_4 - k_2)$ agree. This contradicts the assumption that $j_1$ is the smallest of all integers with the property that $(j_1,j_2) \in \calC_{m+1}$ and $(j_1,j_4) \in \calC_r$ for some $j_2$ and $j_4$. Hence, 
$(j_1,j_4)$ is the root of $\calC_r$.

Recall that $(r_1,r_2)$ is the root of $\calC_{m+1}$ and $j_1 > r_1$. This implies that $(j_1,j_4)$ is not the minimal root in the complement of $\calC_1 \cup \dots \cup \calC_{r-1}$.  This contradiction completes the proof that $\calC_{m+1}\perp \calC_r$ for all $1\leq r \leq m$.

As we said, the inductive procedure stops when $\bs$ is a permutation of $\{1,2,\dots, n_1\}$. Let $\calC_1, \calC_2, \dots, \calC_{m_1}$ be the family of all chains constructed in the induction process when it is stopped. 

Let $\ell$ be the length of the longer of the two diagonals of $Z$ ($L$ is not necessarily the longest diagonal). Note that all projections of $Z$ on any line have length $\ell$ or less. 
Hence, $\Leb(\Pi_\alpha \calM_{n_1, \bs, k} (Z)) \leq \ell/n_1$, for all $k$.

Fix an arbitrarily small $\eps>0$ and choose $n_2$ so large that $\ell/n_2 < \eps /8$.

Let $\calD_1$ be the set of all $(j,k)$ such that $1\leq j,k \leq n_1(1-\eps/(16\ell))$ and let
$\calD_2$ be the set of all $(j,k)$ such that $1\leq j,k \leq n_1$ and $(j,k) \notin \calD_1$.

We relabel $\calC_1, \calC_2, \dots, \calC_{m_1}$ in such a way that all chains $\calC_1, \calC_2, \dots, \calC_{m_2}$ have their roots in $\calD_1$ and the chains $\calC_{m_2+1}, \dots, \calC_{m_1}$
have roots in $\calD_2$ (one of these families may be empty, in principle).
Recall that chains $\calC_r$ are maximal. This implies that 
$n_1$ can be made so large that every chain $\calC_1, \calC_2, \dots, \calC_{m_1}$ has at least $n_2$ elements.
Since $\bs$ is a bijection, the number of elements in 
$\calC_{m_2+1}\cup \dots\cup \calC_{m_1}$
is less than $\eps n_1 /(8\ell)$.

Let $N_1$ be the set of all $j_1$ such that $(j_1,j_2)\in \calC_k$, for some $1\leq  j_2 \leq n_1$ and $1\leq k \leq m_2$, and let $N_2 = \{1,2,\dots, n_1\}\setminus N_1$. 

Recall that $\Leb(\Pi_\alpha \calM_{n_1, \bs, k} (Z)) \leq \ell/n_1$, for all $k$. We have, 
\begin{align}\label{a3.2}
\Leb\left(\Pi_\alpha \left(\bigcup_{k\in N_2}\calM_{n_1, \bs, k} (Z)\right)\right) \leq (\eps n_1 /(8\ell)) (\ell/n_1) = \eps/8.
\end{align}

The key observation is that for every chain $\calC_k$ and all $(j_1,j_2), (j_3, j_4 ) \in \calC_k$, the projections $\Pi_\alpha \calM_{n_1, \bs, j_1} (Z)$
and $\Pi_\alpha \calM_{n_1, \bs, j_3} (Z)$
are identical. More generally,  if $(j_1,j_2) \in \calC_k$ then $\Pi_\alpha (\bigcup_j \calM_{n_1, \bs, j} (Z) ) = \Pi_\alpha  \calM_{n_1, \bs, j_1} (Z) $, where the union is over $j$ such that for some $r$, $(j,r) \in \calC_k$. 
Since every chain $\calC_k$ with $k\leq m_2$ has length at least $n_2$, we see that 
\begin{align*}
\Leb\left(\Pi_\alpha \left(\bigcup_{k\in N_1}\calM_{n_1, \bs, k} (Z)\right)\right)
&\leq (1/n_2) \sum_{k\in N_1}
\Leb\left(\Pi_\alpha \calM_{n_1, \bs, k} (Z)\right)\\
&\leq (1/n_2)  N_1 (\ell/n_1) \leq (1/n_2) \ell < \eps/8.
\end{align*}
Combining this with \eqref{a3.2}, we obtain
\begin{align}\label{a3.3}
\Leb\left(\Pi_\alpha \calM_{n_1, \bs} (Z)\right) < \eps/4.
\end{align}

Next, we consider the case of ``irrational'' $\alpha$ such that 
$\alpha-\pi/2\ne 0,\theta$ and $K_{\alpha-\pi/2}$
does not pass through a point $z\in \calL_{n,\theta}$, $z\ne (0,0)$.
Fix any integer $n\geq 1$.
Note that the line $K_{\alpha-\pi/2}$ passes arbitrarily close to some points
$z\in \calL_{n,\theta}$, $z\ne (0,0)$.
Hence, for any $\eps_1>0$, there exists $z\in \calL_{n,\theta}$
such that $\Leb\left(\Pi_\alpha (Z/n \cup (Z/n + z))\right) < 
\Leb\left(\Pi_\alpha (Z/n )\right)+\eps_1$.
In view of this, it is easy to see that for any integer $n_3$ 
there exists $z\in \calL_{n,\theta}$
such that\begin{align*}
\Leb\left(\Pi_\alpha \left(Z/n \cup \bigcup_{1\leq j \leq n_3} (Z/n + jz )\right)\right) < 
2 \Leb\left(\Pi_\alpha (Z/n )\right).
\end{align*}
We can now repeat the argument given in the case of ``rational'' $\alpha$, with $z$ as in the last formula
and $ n_1 \ge n_3 $. We conclude that 
for any $\alpha$ with $\alpha-\pi/2\ne 0,\theta$ there exist $n_1$ and $\bs$ such that \eqref{a3.3} holds, except for $\eps/4$ replaced by $\eps/2$, that is,
\begin{align}\label{a12.1}
\Leb\left(\Pi_\alpha \calM_{n_1, \bs} (Z)\right) < \eps/2.
\end{align}

Since $\calM_{n_1, \bs} (Z)$ is a finite union of rhombuses, the function 
$\alpha \to \Leb\left(\Pi_\alpha \calM_{n_1, \bs} (Z)\right) $ is continuous. This and \eqref{a12.1} imply that for every $\alpha$ with $\alpha-\pi/2\in (0,\theta) \cup (\theta, \pi)$ there exist $n_1$, $s$ and $\Delta \alpha >0$ such that $\pi/2,\theta+\pi/2 \notin (\alpha - \Delta \alpha, \alpha + \Delta \alpha)$ and for $\beta \in (\alpha - \Delta \alpha, \alpha + \Delta \alpha)$, 
\begin{align}\label{a3.4}
\Leb\left(\Pi_\beta \calM_{n_1, \bs} (Z)\right) < \eps.
\end{align}

The set of $\alpha$ satisfying $\alpha -\pi/2 \in [\eps, \theta-\eps] \cup [\theta+\eps, \pi -\eps]$ 
is compact so it is covered by a finite number of intervals of the form 
$(\alpha - \Delta \alpha, \alpha + \Delta \alpha)$. Let $\alpha_1, \alpha_2, \dots, \alpha_m$ be the set of centers of these intervals and let $(j_n, \bs_n)$ be the integer and permutation such that \eqref{a3.4} holds with $n_1 = j_n$, $\bs = \bs_n$ and 
$\beta \in (\alpha_n - \Delta \alpha_n, \alpha_n + \Delta \alpha_n)$.
By \eqref{a3.7} and \eqref{a3.4}, for $\alpha$ satisfying $\alpha -\pi/2 \in [\eps, \theta-\eps] \cup [\theta+\eps, \pi -\eps]$,
\begin{align}\label{a3.8}
\Leb(\Pi_\alpha 
(\calM_{j_1,\bs_1} \circ\calM_{j_2,\bs_2}\circ \cdots
\circ\calM_{j_m,\bs_m}(Z) )) < \eps.
\end{align}
According to Remark \ref{a3.6} (ii), there exist $j_*$ and $\bs_*$ such that
\begin{align}\label{a3.9}
\calM_{j_1,\bs_1} \circ\calM_{j_2,\bs_2}\circ \cdots
\circ\calM_{j_m,\bs_m}(Z) = \calM_{j_*,\bs_*}(Z).
\end{align}
We now take $G = \calM_{j_*,\bs_*}(L)$. Since $L\subset Z$, we obtain from \eqref{a3.8} and \eqref{a3.9},
 for $\alpha$ satisfying $\alpha -\pi/2 \in [\eps, \theta-\eps] \cup [\theta+\eps, \pi -\eps]$,
\begin{align*}
\Leb(\Pi_\alpha G) = \Leb(\Pi_\alpha \calM_{j_*,\bs_*}(L))
\leq \Leb(\Pi_\alpha \calM_{j_*,\bs_*}(Z)) < \eps.
\end{align*}
This completes the proof of Theorem \ref{a3.1}.

\bibliographystyle{alpha}
\bibliography{invisible}

\end{document}